\documentclass[12pt, reqno]{amsart}
\usepackage{amsmath}  
\usepackage{amscd}
\usepackage{amsthm} 
\usepackage{amssymb} \usepackage{latexsym}
\usepackage{euscript}
\usepackage{epsfig}  
\usepackage{graphics}   
\usepackage{array}
\usepackage{enumerate}
\usepackage{color}
\usepackage{wasysym} 
\usepackage{pdfsync} 
\usepackage{stmaryrd}
\usepackage{tikz}
\usepackage{tkz-euclide} 
\usetikzlibrary{fit}

\newcommand{\bel}[1]{\begin{equation}\label{#1}}

\newcommand{\be}{\begin{equation}}

\newcommand{\ba}{\begin{eqnarray}}
\newcommand{\ea}{\end{eqnarray}}

\newcommand{\qe}{\end{equation}}
\newcommand{\R}{{\mathbb R}}
\newcommand{\N}{{\mathbb N}}

\newcommand{\wt}{\widetilde}

\newcommand{\dist}{\rho}

\newcommand{\eg}{\begin{example}}
\newcommand{\egd}{\end{example}}
\newcommand{\tm}{\begin{thm}}
\newcommand{\tmd}{\end{thm}}
\newcommand{\co}{\begin{coro}}
\newcommand{\cod}{\end{coro}}
\newcommand{\enu}{\begin{enumerate}}
\newcommand{\enud}{\end{enumerate}}
\newcommand{\rmk}{\begin{rem}}
\newcommand{\rmkd}{\end{rem}}

\newtheorem{thm}{Theorem}[section]
\newtheorem{prop}{Proposition}[section]
\newtheorem{example}{Example}[section]
\newtheorem{coro}{Corollary}[section]
\newtheorem{lemma}{Lemma}[section]
\theoremstyle{definition}
\newtheorem{defi}{Definition}[section]
\theoremstyle{definition}
\newtheorem{assump}{Assumption}[section]

\theoremstyle{remark}
\newtheorem{rem}{Remark}[section]
\theoremstyle{remark}

\usepackage{color}
\usepackage{changes}
\usepackage{enumerate}
\usepackage{ulem,ifthen,xcolor,xkeyval,pdfcolmk}

\definechangesauthor[color=blue]{bh}
\definechangesauthor[color=red]{zl}

\begin{document}

\title[Graphs with positive spectrum]{Graphs with positive spectrum}

\author{Bobo Hua}
\address{Bobo Hua: School of Mathematical Sciences, LMNS,
Fudan University, Shanghai 200433, China; Shanghai Center for
Mathematical Sciences, Fudan University, Shanghai 200433,
China.}
\email{bobohua@fudan.edu.cn}

\author{Zhiqin Lu}
\address{Zhiqin Lu: Department of Mathematics,
University of California,
410D Rowland Hall,
Irvine, CA 92697-3875,
USA.}
\email{zlu@uci.edu}
\begin{abstract}
In this paper, we prove sharp $\ell^2$ decay estimates of nonnegative generalized subharmonic functions on graphs with positive Laplacian spectrum, which extends the result by Li and Wang \cite{LiWang01} on Riemannian manifolds. \end{abstract}

\maketitle





\section{Introducton}


Let $M$ be a complete, noncompact Riemannian manifold without boundary. For any $R>0,$ we write $B_R(x)$ for the ball of radius $R$ centered at $x,$ and $\mathrm{vol}(\,\cdot\,)$ for the Riemannian volume.
We denote by 
$\mu_1(M)$ the bottom of the spectrum of the Laplace-Beltrami operator on $M,$ and by $\mu_e(M)$ the bottom of its essential spectrum.

In 1975, Cheng and Yau \cite{ChengYau75} proved that if $M$ has polynomial volume growth, then $\mu_1(M)=0.$  
Generalizing Cheng and Yau's result, Brooks \cite{Brooks1981,Brooks1984} proved that
\begin{equation}\label{eq:brooks}\mu_e(M)\leq \frac{1}{4}\tau(M)^2,\end{equation}
where the volume entropy of $M$ is defined as
 $$\tau(M):=\left\{\begin{array}{ll}{\displaystyle \limsup_{R\to\infty}}\,\frac{1}{R}\log \mathrm{vol}(B_R(p)),& \mathrm{if}\ \mathrm{vol}(M)=\infty,\\[1ex]
 {\displaystyle \limsup_{R\to\infty}}\,(-\frac{1}{R}\log \mathrm{vol}(M\setminus B_R(p))),& \mathrm{if}\ \mathrm{vol}(M)<\infty.\\ \end{array}\right.$$

For any compact $\Omega\subset M,$ any non-compact connected component $\Pi$ of $M\setminus \Omega$ is called an end of $M$ with respect to  $\Omega.$ We denote by $\mu_1(\Pi)$ the bottom of the spectrum of the Laplace-Beltrami operator on $\Pi$ with Dirichlet boundary condition. Fix $p\in M,$ we denote by $r(x)=d(x,p)$ the Riemannian distance function to the point $p.$ We write $\Pi_R:=\Pi\cap B_R(p)$ for $R>0$ and $$\Pi_{R_1}^{R_2}:=\Pi\cap \{x\in M: R_1\leq r(x)\leq R_2\},\quad \mathrm{for}\ 0<R_1<R_2.$$
We say that $f$ is a \emph{generalized subharmonic function} on $M$ if $\Delta f\geq  -\mu f$ for some $\mu\in \R.$

In~\cite[Ch.22]{Li12},
Li and Wang \cite{LiWang01} proved the following theorem on the $L^2$-decay estimates of generalized subharmonic functions,  which improved the above results of Brooks.

 \begin{thm}[Li-Wang] Let $M$ be a complete Riemannian manifold. {Let $p\in M$ be fixed point.} Let $R_0>0$ be a real number. Suppose $\Pi$ is an end of $M$ with respect to $B_{R_0}(p)$ such that $\mu_1(\Pi)>\mu$ for some constant $\mu> 0$. Let $f$ be a nonnegative generalized subharmonic function defined on $\Pi,$ 
 \[
 \Delta f\geq -\mu f.
 \]
 If $f$ satisfies the growth condition
 \begin{equation}\label{eq:con1}
 \int_{\Pi_R} f^2 e^{-2a r}=o(R),\quad R\to\infty
 \end{equation}
 with
 \[
 a=\sqrt{\mu_1(\Pi)-\mu},
 \]
 then it  satisfy the decay estimate
 \[
 \int_{\Pi_\rho^{\rho+1}} f^2
 \leq \frac{2a+1}{a^2} e^{-2a\rho}
  \int_{\Pi^{R_0+1}_{R_0}} e^{2a r}f^2
  \]
 for all $\rho\geq 2(R_0+1)$.  \end{thm}

\begin{rem} 
\begin{enumerate}[(i)]\item One can't replace the assumption of $o(R)$ by $O(R)$, which can be seen by considering a non-constant bounded harmonic function $f$ on the hyperbolic space, see e.g. \cite[p.277]{Li12}.
\item By this theorem, one can derive a stronger result than Brooks' result \eqref{eq:brooks} using a variant of $\tau(M)$ defined by replacing the limsup in the definition of $\tau(M)$ by the liminf, see e.g. \cite{HKW13JLMS} or Corollary~\ref{coro:app3} below.
\end{enumerate} 
\end{rem}

The purpose of this paper is to extend Li-Wang's result to graphs with positive spectrum. 
In order to reach the goal, we give a new proof of the Li-Wang estimate by introducing new test functions, which avoids some technical estimates in \cite{LiWang01}. This is {then} robust enough to be extended to the setting of graphs.

 We recall the setting of weighted graphs. Let $(V,E)$ be a locally finite, simple, undirected graph with the set of vertices $V$ and the set of edges $E.$ Two vertices $x,y$ are called neighbours, denoted by $x\sim y$, if there is an edge connecting $x$ and $y,$ i.e. $\{x,y\}\in E.$ 
 A graph is called connected if for any $x,y\in V,$ there are vertices $z_i,$ $0\leq i\leq n$, such that $x = z_0 \sim. . . \sim z_n = y.$ We denote by $$d(x, y) =
\min\{n | x = z_0 \sim. . . \sim z_n = y\}$$ the combinatorial distance between $x$ and $y$, that is, the minimal number of edges in a path among all paths connecting $x$ and $y.$ We always assume that the graph $(V,E)$ is connected.  
Let $$w: E\to \R_+,\ \{x,y\}\mapsto w_{xy}=w_{yx}$$ be an edge weight function, and $$m:V\to \R_+,\ x\mapsto m_x$$ be a vertex weight function. For convenience, we extend $w$ to $V\times V$ by assigning $w_{xy}=0$ to the pair $(x,y)$ with $x\not\sim y.$ We denote by $$|\Omega|:=\sum_{x\in \Omega}m(\Omega),\quad \Omega\subset V$$ the $m$-measure of $\Omega,$ and by $\ell^p(V,m),$ $p\in [1,\infty],$ the space of $\ell^p$-summable functions on $V$ with respect to $m.$ For any function $f:V\to\R,$ we write
\begin{equation}\label{abc}
\sum_{\Omega} f m=\sum_{x\in \Omega} f(x) m_x,
\end{equation}
 whenever it makes sense. We call the quadruple $G=(V,E,m,w)$ a weighted graph. For a weighted graph $G,$ the Laplace operator $\Delta$ is defined as, for any function $f:V\to \R,$
$$\Delta f (x):= \sum_{y\in V}\frac{w_{xy}}{m_x}\left(f(y)-f(x)\right), \quad\forall x\in V.$$
A function $f$ on $V$ is called harmonic if $\Delta f=0.$


For the analysis on general weighted graphs, Frank, Lenz and Wingert \cite{FLW12} introduced the so-called intrinsic metrics, see e.g. \cite{GHM,KellerLenz12,HKMW13,HKW13JLMS,BHKadv,Huang14pa,Folz14,HShiozawa,HuaKeller14,Bauer-Keller-Wojciechowski15,HuaLin17,BauerHY17,GLLY18} for recent developments. A \emph{(pseudo) metric} is a map $\rho:V\times V\to [0,\infty),$ which is symmetric, satisfies the triangle inequality and $\rho(x,x)=0$ for all $x\in V.$  For any $R\in \R_+:=(0,\infty),$ we write $B_R(x):=\{y\in V: \dist(y,x)\leq R\}$ for the ball of radius $R$ centered at $x.$ 
A metric $\dist$ is called an \emph{intrinsic} metric on $G$, if for any $x\in V,$
$$\sum_{y\in V}w_{xy}\dist^2(x,y)\leq m_x.$$ We denote by
$$s:=\sup_{x\sim y}\dist(x,y)$$ the jump size of the metric $\dist.$ In this paper, we only consider intrinsic metrics satisfying the following assumption.
\begin{assump}\label{ass:1} The metric $\dist$ is an intrinsic metric on $G$ such that
\begin{enumerate}[(i)]
\item for any $ x\in V, R\in \R_+,$ $B_R(x)$ is a finite set, and
\item $\dist$ has finite jump size, i.e. $s<\infty.$
\end{enumerate}
\end{assump}

Note that for the combinatorial distance $d,$ the jump size $s=1.$
The combinatorial distance $d$ is an intrinsic metric on $G$ if and only if
$$\sum_{y\in V}w_{xy}\leq m_x,\quad \forall x\in V.$$ In the following, for $\rho=d,$ we always set $m_x:=\sum_{y\in V}w_{xy},$ and the corresponding Laplacian is called normalized Laplacian in the literature.

For any $\Omega\subset V,$ 
we denote by
$\mu_1(\Omega)$ the bottom of the spectrum of the Laplacian with Dirichlet boundary condition on $\Omega,$ and by $\mu_1(G):=\mu_1(V)$ the bottom of the spectrum of the Laplacian on $G.$ We write $C_0(\Omega)$ for the set of functions of finite support, whose support is contained in $\Omega.$
By the Rayleigh quotient characterization, $$\mu_1(\Omega)=\inf_{f\in C_0(\Omega)\setminus\{0\}}\frac{\frac12\sum_{x,y\in V}w_{xy}|f(x)-f(y)|^2}{\sum_{x\in V}f^2(x)m_x}.$$ Furthermore, we denote by $\mu_e(G)$ the bottom of the essential spectrum of the Laplacian on $G.$

Fix a vertex $x_0.$  We denote by $r(x)=\dist(x,x_0)$ the distance function to $x_0,$ and by $$A_{R_1}^{R_2}:=\{x\in V: R_1\leq r(x)\leq R_2\}$$ the annulus of inner radius $R_1$ and outer radius $R_2$ for $0<R_1<R_2.$ For any subset $K\subset V,$ we write $\partial K=\{y\in V\setminus K: \exists x\in K, y\sim x\}$ for the vertex boundary of $K,$ and $\overline{K}=K\cup \partial K.$ For any finite $\Omega\subset V,$ we consider $V\setminus \Omega$ as the induced subgraph on $V\setminus \Omega.$ Any infinite connected component $\Pi$ of $V\setminus \Omega$ is called an end of $G$ with respect to $\Omega.$ In general, when we say that $\Pi$ is an end we mean that $\Pi$ is an end of $G$ with respect to some finite $\Omega.$ We write $\Pi_R=\Pi\cap B_R(x_0)$ for $R>0$ and $$\Pi_{R_1}^{R_2}:=\Pi\cap A_{R_1}^{R_2},\quad \mathrm{for}\ 0<R_1<R_2.$$

In this paper, we always use the following definition for the function $a(\cdot).$
\begin{defi}\label{def:d1}Let $G=(V,E,m,w)$ be an infinite weighted graph, and $\dist$ be an intrinsic metric satisfying Assumption~\ref{ass:1} with jump size $s$. Given any $\mu\in \R,$ we define

\[a_\mu(t):=\left\{\begin{array}{ll}\frac{1}{s}\log (1+s\sqrt{2(t-\mu)}),\ \forall\ t\in(\mu,+\infty),& \mathrm{if}\ \dist\neq d,\\
\log\frac{1-\mu}{1-t}+\log\left(1+\sqrt{1-\left(\frac{1-t}{1-\mu}\right)^2}\right),\ \forall\ t\in(\mu,1),& \mathrm{if}\ \dist=d.\\ \end{array}\right.\]

\end{defi}
\begin{rem} For our purposes, we distinguish two cases, whether $\dist$ is the combinatorial distance $d$ or not. Note that $a=a_\mu(t)$ satisfies $$\frac{(e^{as}-1)^2}{2s^2}=t-\mu,\quad \mathrm{for}\  \dist\neq d,\ \quad\mathrm{and}$$ 
 $$\frac{(e^a-1)^2}{1+e^{2a}}=\frac{t-\mu}{1-\mu}, \quad \mathrm{for}\ \dist=d.$$ The latter equation for $\mu=0$ first appears in Fujiwara's paper \cite{Fujiwaragrowth}.
\end{rem}

In this paper, we prove a discrete analog of Li and Wang's result.
\begin{thm}\label{thm:main1} Let $G=(V,E,m,w)$ be an infinite weighted graph, and $\dist$ be an intrinsic metric  satisfying Assumption~\ref{ass:1}. Let $\Pi$ be an end with $\mu_1:=\mu_1(\Pi).$ Let $f$ be a nonnegative function satisfying $\Delta f\geq-\mu f$ for some $ \mu< \mu_1.$ Let $a=a_\mu(\mu_1),$ see Definition~\ref{def:d1}. If there exist $R_i\to \infty,$ $i\to \infty,$ such that
\begin{equation}\label{eq:decay1}\sum_{\Pi_{R_i}^{R_i+3s}}f^2e^{-2ar}m\to 0,\quad i\to \infty,\end{equation}
then for $R_0\geq \max_{x\in \partial \Pi}r(x)$ and $R\geq R_0+3s,$
$$\sum_{\Pi_{R}^{R+3s}} f^2 m\leq Ce^{-2aR}\sum_{\Pi_{R_0}^{R_0+3s}}f^2 e^{2ar}m,$$ where $C=\frac{7e^{10as}}{s^2(\mu_1-\mu)}.$
\end{thm} 

\begin{rem} \begin{enumerate}[(i)]
\item The condition in \eqref{eq:decay1} is weaker than the following condition,
\begin{equation}\label{eq:disa1}\sum_{\Pi_R}f^2 e^{-2ar}m=o(R),\ R\to \infty,\end{equation} which is a discrete analog of \eqref{eq:con1}, see Proposition~\ref{prop:weaker}.
\item For $\rho=d,$ by Proposition~\ref{prop:less1}, $\mu_1<1,$ which implies that $a_\mu(\mu_1)$ is well-defined.
\item For $\rho=d,$ the result is sharp for homogeneous trees, see Example~\ref{ex:tree1}.\end{enumerate}
\end{rem}
On Riemannian manifolds, for the proof of the case $f\equiv 1,$ Brooks \cite{Brooks1981} introduced a test function which behaves as $e^{a r(x)}$ for $r(x)\leq L,$  $L>0,$ and $e^{-a r(x)}$ for $r(x)\geq L.$ There are many generalizations of Brooks' result on graphs, see e.g.  \cite{DoKa88,OU1994,Fujiwaragrowth,HKW13JLMS}. In particular, Haeseler, Keller, and Wojciechowski \cite{HKW13JLMS} proved Brooks' result for general weighted graphs (with possibly unbounded Laplacians) admitting intrinsic metrics. For Riemannian manifolds, Li and Wang \cite{LiWang01} introduced suitable test functions to extend the results for generalized subharmonic function $f.$ It seems that Li and Wang's test functions don't work in the graph setting. In particular, they derived a useful positive term from a mixed term, see \cite[p.9]{LiWang01} or \cite[p.272]{Li12}. 
 In this paper, we introduce new test functions, see Section~2 or \eqref{def:ph} for the definition, to avoid the mixed term estimate, and are able to prove $\ell^2$ decay estimates for generalized subharmonic functions on weighted graphs admitting intrinsic metrics.



In the following, we discuss some applications of the main result. For any $\alpha\geq 0,$ we denote by $$g_\alpha(x,y)=\int_0^\infty e^{-\alpha t}p_t(x,y)dt$$ the $\alpha$-resolvent kernel, i.e. the kernel of the operator $(-\Delta+\alpha)^{-1},$ where $p_t(\cdot,\cdot)$ is the minimal heat kernel on $G$. In particular, for $\alpha=0,$ $g_0(\cdot,\cdot)$ is the minimal Green's function on $G,$ which is finite if and only if $G$ is non-parabolic (also known as transient). By Theorem~\ref{thm:main1}, we prove the following result.
\begin{thm}\label{thm:app1} Let $G=(V,E,m,w)$ be an infinite weighted graph with the bottom of the spectrum $\mu_1(G)>0,$ and $\rho$ be an intrinsic metric satisfying Assumption~\ref{ass:1}. For $\alpha\geq 0,$ let $a=a_{-\alpha}(\mu_1(G)).$ Then  
$$\sum_{A_{R}^{R+3s}}g_{\alpha}^2(x_0,\cdot)m\leq C e^{-2aR},\quad \forall R>0,$$ where the constant $C$ is independent of $R.$
\end{thm} 

\begin{example}\label{ex:tree1} For any integer $N\geq 3,$ let $T_N$ be the $N$-regular tree, i.e. the homogeneous tree of degree $N.$ Consider the standard weights, $w_{xy}=1$ for $\{x,y\}\in E,$ $m_x=N$ for all $x\in V,$ and let $\dist=d.$ Let $x_0$ be the root of the tree.
Then $$\mu_1(T_N)=\mu_e(T_N)=1-\frac{2\sqrt{N-1}}{N}.$$  For any $\alpha\geq 0,$ any vertex $x$ with $d(x,x_0)=n,$
$$g_\alpha(x_0,x)=\frac{1}{N(\alpha+1-\frac1b)}b^{-n},$$ where $$b=\frac12((\alpha+1)N+\sqrt{(\alpha+1)^2N^2-4N+4}).$$ In particular, $b$ is the solution to the following equation $b+\frac{N-1}{b}=(\alpha+1)N.$ Hence,
$$\sum_{A_{n}^{n+3}}g_{\alpha}^2(x_0,\cdot)m\sim C(N,\alpha) \left(\frac{N-1}{b^2}\right)^n,\quad n\to \infty.$$ Moreover, $(\frac{N-1}{b^2})^n=e^{-2an},$ where $a=a_{-\alpha}(\mu_1(T_N)).$ Therefore, the $\ell^2$ decay estimate in Theorem~\ref{thm:app1} (and that in Thereom~\ref{thm:main1}) is sharp.

\end{example}

\begin{defi}\label{def:para1} An end $\Pi$ is called non-parabolic if there exists $f:\overline{\Pi}\to \R_+$ satisfying $\Delta f=0$ on $\Pi,$ $f\big|_{\partial \Pi}\equiv 1$ and $$\liminf_{x\to \Pi(\infty)}f(x)<1,$$ where $x\to \Pi(\infty)$ is understood as $x\to \infty$ and $x\in \Pi.$ Here $f$ is called a barrier function on $\Pi.$ Otherwise, $\Pi$ is called parabolic.
\end{defi}

We have the following corollaries.
\begin{coro}\label{thm:volume} Let $G=(V,E,m,w)$ be an infinite weighted graph, and $\rho$ be an intrinsic metric satisfying Assumption~\ref{ass:1}. Let $\Pi$ be an end satisfying $\mu_1(\Pi)>0$ and $a=a_0(\mu_1(\Pi)).$ Let $R_0\geq \max_{x\in \partial \Pi}r(x).$ Then the following are equivalent:

\begin{enumerate}[(i)]
\item $\Pi$ is a parabolic end.
\item There exist $R_i\to \infty,$ $i\to \infty,$ such that
$$\sum_{\Pi_{R_i}^{R_i+3s}}e^{-2ar}m\to 0,\quad i\to \infty.$$
\item $\Pi$ has finite total volume, i.e. $|\Pi|<\infty.$
\item  
For any $R\geq R_0+3s,$
$$|\Pi_R^{R+3s}|\leq Ce^{-2a(R-R_0)}|\Pi_{R_0}^{R_0+3s}|.$$
\item For any $R\geq R_0+3s,$ $$|\Pi|-|\Pi_R|\leq Ce^{-2a(R-R_0)}|\Pi_{R_0}^{R_0+3s}|.$$
\end{enumerate}
\end{coro}

\begin{coro}\label{thm:volume2} Let $G=(V,E,m,w)$ be an infinite weighted graph, and $\dist$ be an intrinsic metric satisfying Assumption~\ref{ass:1}. Let $\Pi$ be an end satisfying $\mu_1(\Pi)>0$ and $a=a_0(\mu_1(\Pi)).$ Then the following are equivalent:

\begin{enumerate}[(i)]
\item $\Pi$ is a non-parabolic end.
\item There exist positive constants $R_0$ and $\epsilon_0$ such that for any $R\geq R_0,$
$$\sum_{\Pi_{R}^{R+3s}}e^{-2a r}m\geq \epsilon_0.$$
\item $\Pi$ has infinite total volume, i.e. $|\Pi|=\infty.$
\item  
For sufficiently large $R,$
$$|\Pi_R^{R+3s}|\geq Ce^{2aR},$$ where $C$ depends on $\Pi.$
\end{enumerate}
\end{coro}
\begin{rem} Let $\Pi$ be an end of the $N$-regular tree $T_N$ with respect to the root $\{x_0\},$ see Example~\ref{ex:tree1}. Let $\dist=d.$ Note that
$$|\Pi_n^{n+3}|\sim C (N-1)^n,\quad n\to \infty.$$ We know that $\mu_1(\Pi)=1-\frac{2\sqrt{N-1}}{N},$ which yields that $$a_0(\mu_1(\Pi))=\frac12\log(N-1).$$ Hence the volume growth condition of $(iv)$ in Corollary~\ref{thm:volume2} is sharp.
\end{rem}


We define the (modified) volume entropy of $G$ as $$\wt{\tau}(G):=\left\{\begin{array}{ll}{\displaystyle \liminf_{R\to \infty}}\,\frac{1}{R}\log|B_R(x_0)|,& \mathrm{if}\ |V|=\infty,\\
{\displaystyle  \liminf_{R\to \infty}}(-\frac{1}{R}\log |V\setminus B_R(x_0)|),&\mathrm{if}\  |V|<\infty.\\ \end{array}\right.$$ Note that $\wt{\tau}(G)$ is a variant of $\tau(M)$ for a manifold.
 

The following corollary is a discrete generalization of Brooks' results \cite{Brooks1981,Brooks1984}. For graphs with infinite volume, see analogous results in \cite{DoKa88,OU1994,Fujiwaragrowth,HKW13JLMS}. 
\begin{coro}\label{coro:app3}  Let $G=(V,E,m,w)$ be an infinite weighted graph with $\mu_e(G)>0,$ and $\rho$ be an intrinsic metric satisfying Assumption~\ref{ass:1}. For any $0<\epsilon<\mu_e(G),$ let $a=a_0(\mu_e(G)-\epsilon).$ Then for any sufficiently large $R,$
 $$\begin{array}{ll}|B_R(x_0)|\geq Ce^{2aR },& \mathrm{if}\ |V|=\infty,\ \mathrm{and}\\
 |V\setminus B_R(x_0)| \leq Ce^{-2aR},& \mathrm{if}\ |V|<\infty.\\ \end{array}$$
In particular, \begin{equation}\label{eq:be1}\mu_e(G)\leq a_0^{-1}\left(\frac{\wt{\tau}(G)}{2}\right),\end{equation} where $a_0^{-1}(\cdot)$ is the inverse function of $a_0(\cdot).$


\end{coro}
\begin{rem} 
\begin{enumerate}[(i)]
\item The above result for graphs with finite volume is new.
\item For $\rho=d,$ $a_0^{-1}(x)=\frac{(e^{x}-1)^2}{1+e^{2x}}.$
Consider the homogeneous tree $T_N,$ see Example~\ref{ex:tree1}. 
Note that $\wt{\tau}(T_N)=\log(N-1)$ and $\mu_e(T_N)=1-\frac{2\sqrt{N-1}}{N}.$ Hence, the inequality \eqref{eq:be1} is in fact an equality in this case. This was also obtained by \cite{Fujiwaragrowth}.
\end{enumerate}
\end{rem}

Moreover, we prove that any $\ell^q$-summable ($q\geq 2$) harmonic function on an end $\Pi$ with $\mu_1(\Pi)>0$ is bounded, and in fact it has exponential decay on the end, see Corollary~\ref{coro:app4}.

The paper is organized as follows: In the next section, we give a new proof of the Li-Wang estimate. In Section~3, we  recall some basic properties of graphs. In Section~\ref{sec:proof}, we prove the main result, Theorem~\ref{thm:main1}. Section~\ref{sec:app} is devoted to various applications. In this paper, for simplicity the constant $C$ may change from line to line.

\section{A new proof of Li-Wang's result}
In this section, we give a simplified proof of an inequality of Li and Wang (\cite{LiWang01}).

\begin{thm}\label{thm:LiWang}Let $M$ be a complete Riemannian manifold. Suppose that $\Pi$ is an end of $M$ with respect to $B_{R_0}(p)$ such that $\mu_1(\Pi)>\mu$ for some constant $\mu.$ Let $a=\sqrt{\mu_1(\Pi)-\mu}>0.$ Let $f$ be  a nonnegative  function satisfying 
\begin{equation}\label{1}
\Delta f\geq  -\mu f.
\end{equation}
 If there is a sequence $R_j\to\infty$ such that 
\begin{equation}\label{eq:gc1}\int_{\Pi_{R_j}^{R_j+1}} e^{-2ar} f^2\to 0,\end{equation} then 
$$\int_{\Pi_\rho^{\rho+1}}f^2\leq \frac{2a+1}{a^2}e^{-2a\rho}\int_{\Pi_{R_0}^{R_0+1}}e^{2ar}f^2$$
for $\rho>R_0+1$. \end{thm}
\begin{rem} \begin{enumerate}[(i)]
\item The condition \eqref{eq:gc1} is weaker than \eqref{eq:con1}, see e.g. Proposition~\ref{prop:weaker}.
\item Li and Wang \cite{LiWang01} used a mixed term, the second term on the right hand side of \eqref{eq:lwe}, $\int_\Pi\phi\, e^{2h}\langle \nabla\phi,\nabla h\rangle f^2,$ to extract a positive term and derive the desired estimate, see \cite[p.9]{LiWang01} or \cite[p.272]{Li12}. In our proof, via introducing new test functions, we may get rid of such estimates for the mixed term.
\end{enumerate}
\end{rem}
  
 \begin{proof}
 We consider the equality
 \[
 \int_\Pi|\nabla (\phi f)|^2=\int_\Pi|\nabla\phi|^2 f^2-\int_\Pi\phi^2 f\Delta f,
 \]
 where $\phi$ is a Lispchitz function with compact support in $\Pi$. 
 By~\eqref{1} and the fact that 
 \[
  \int_\Pi|\nabla (\phi f)|^2\geq\mu_1(\Pi)\int_\Pi(\phi f)^2,
  \]
we have
 \[
a^2\int_\Pi\phi^2 f^2\leq \int_\Pi|\nabla\phi|^2 f^2.
\]

 Now if we replace $\phi$ by $\phi \,e^h$ for function $h$ to be defined below,  then the above inequality is turned into  the inequality 
considered in Li-Wang's paper.
Expanding $\nabla(\phi\,e^h)$, we have 

 \begin{equation}\label{eq:lwe}
 a^2\int_\Pi\phi^2 e^{2h} f^2\leq \int_\Pi|\nabla\phi|^2 e^{2h} f^2
 +2\int_\Pi\phi\, e^{2h}\langle \nabla\phi,\nabla h\rangle f^2+\int_\Pi
 \phi^2|\nabla h|^2 e^{2h} f^2.
\end{equation}

Let $\phi,h$ be defined as follows. Let $R_0>0$, and let $R>R_0+1$. Define 
\[
\phi (t)=\left\{
\begin{array}{ll}
0 & 0\leq t< R_0\\
t-R_0 &R_0\leq t< R_0+1\\
1 & R_0+1\leq t<R\\
-t+R+1& R \leq t<R+1\\
0 &t\geq R+1
\end{array}
\right..
\]
 
 \begin{center}
\begin{tikzpicture}
 \tkzInit[xmax=9,ymax=3,xmin=-3,ymin=-3]

  \draw[thick, blue] (1,0)node[below]{$\scriptstyle R_0$}--(2,1);  
    \draw[thick,blue] (2,1)--(6.5,1);  
        \draw[dashed] (2,1)--(2,0);  
               \draw[dashed] (6.5,1)--(6.5,0);  
        \draw[thick,blue] (7.5,0)node [below] {$\scriptstyle R+1$}--(6.5,1);  
 \draw[thick,->] (-2,0)--(9,0) node[below] {x}; 
    \draw[thick,->] (0,-1.5)--(0,2.) node[left] {y}; 
    
 
          \node  at (-.2cm, -.2cm) {$0$};
           \node  at (6.5,0)[below,blue ] {$\scriptstyle R$};
          \node at (2,0) [below,blue]{$\scriptstyle R_0+1$};
\end{tikzpicture}
\end{center}

Let $\rho>R_0+1$ but $\rho<R-1$.
The function $h$ is defined as
\[
h(t)=\left\{
\begin{array}{ll}
at &0\leq t<\rho\\
a\rho &\rho\leq t<\rho+1\\
a\rho-a(t-\rho-1) & t\geq \rho+1
\end{array}
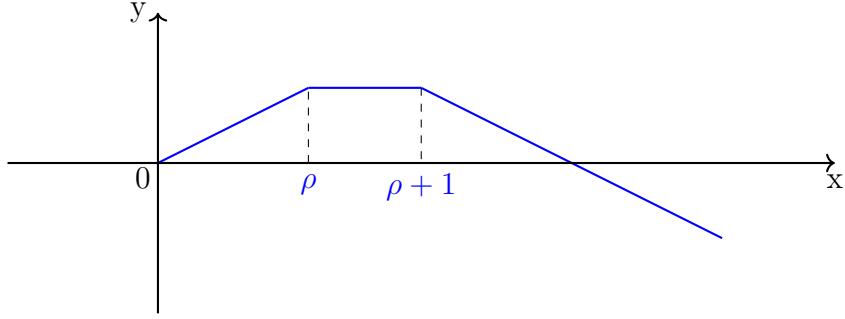
\right..
\]

\begin{center}
\begin{tikzpicture}
 \tkzInit[xmax=10,ymax=7,xmin=-3,ymin=-3]

  \draw[thick, blue] (0,0)--(2,1);  
    \draw[thick,blue] (2,1)--(3.5,1);  
        \draw[dashed] (2,0)--(2,1);  
               \draw[dashed] (3.5,1)--(3.5,0)node [below,blue]{$\rho+1$};  
        \draw[thick,blue] (3.5,1)--(7.5,-1);  
 \draw[thick,->] (-2,0)--(9,0) node[below] {x}; 
    \draw[thick,->] (0,-2.)--(0,2.) node[left] {y}; 
    
 
          \node  at (-.2cm, -.2cm) {$0$};
          
          \node at (2,0) [below,blue]{$\rho$};
\end{tikzpicture}
\end{center}

Since 
\[
\left|\int_\Pi\phi\, e^{2h}\langle \nabla\phi,\nabla h\rangle f^2\right|
\leq a\int_{\Pi_{R_0}^{R_0+1}} e^{2h} f^2+ a\int_{\Pi_{R}^{R+1}} e^{2h} f^2,
\]
we have
\[
\begin{split}
&
a^2\int_\Pi\phi^2 e^{2h} f^2\leq (2a+1)\int_{\Pi^{R_0+1}_{R_0}} e^{2h} f^2+(2a+1)\int_{\Pi_R^{R+1}} e^{2h} f^2\\&+a^2\int_{\Pi\backslash\Pi_{\rho_0}^{\rho_0+1}}\phi^2 e^{2h} f^2.
\end{split}
\]
Thus we have
\[
\begin{split}
&
a^2e^{2a\rho}\int_{\Pi_\rho^{\rho+1}}  f^2\leq  (2a+1)\int_{\Pi_{R_0}^{R_0+1}} e^{2h} f^2+(2a+1)e^{(4a+2)\rho} \int_{\Pi_R^{R+1}}  e^{-2ar} f^2.
\end{split}
\]

 For fixed $\rho$, by the decay condition, there is a sequence $R_j\to\infty$ such that 
 \[
 \int_{\Pi_{R_j}^{R_j+1}} e^{-2ar} f^2\to 0.
 \]
If we take $R=R_j$ and let $j\to\infty$, we get
 
 \[
a^2 e^{2a\rho}\int_{\Pi_\rho^{\rho+1}} f^2\leq
(2a+1)\int_{\Pi_{R_0}^{R_0+1}} e^{-2ar} f^2.\]

\end{proof}







\section{Calculus on graphs}
Let $G=(V,E,m,w)$ be a weighted graph. 
For any functions $f,g:V\to\R,$ we introduce the ``carr\'e du champ" operator $\Gamma$ as follows,
$$\Gamma(f,g)=\frac12(\Delta(fg)-g\Delta f-f\Delta g),\qquad {\Gamma(f)=\Gamma(f,f)}.$$ 
For any $x,y\in V,$ and $f:V\to \R,$ we denote $$\nabla_{xy}f=f(y)-f(x).$$
One easily checks that for any $f,g:V\to \R,$ any $x,y\in V,$
\begin{equation}\label{eq:leibniz}\nabla_{xy}(fg)=f(x)\nabla_{xy}g+g(y)\nabla_{xy}f, 
\end{equation}
{and}
\begin{equation}\label{eq:leibniz2}
\Gamma(f,g)=\sum_{x,y} w_{xy} f(x)g(x)\nabla_{xy} f\nabla_{xy} g.
\end{equation}

The following Green's formula is well-known, see e.g. \cite{Grigoryanbook}.
\begin{thm}[Green's formula] Let $f,g:V\to \R,$ and $g\in C_0(V).$
$$\frac{1}{2}\sum_{x,y\in V}w_{xy}\nabla_{xy}f\nabla_{xy}g=-\sum f\Delta gm,$$
{where for any function $f:V\to\R,$ $\sum fm$ (or $ \sum_x fm$) is defined in~\eqref{abc}.}
\end{thm}

\begin{prop}\label{prop:form1} Let $f,g:V\to \R,$ and $g\in C_0(V).$ Then 
$$\frac{1}{2}\sum_{x,y\in V}w_{xy}|\nabla_{xy} (fg)|^2= \sum f^2\Gamma(g)m-\sum f\Delta f g^2m-\frac14\sum_{x,y}w_{xy}|\nabla_{xy}f|^2|\nabla_{xy}g|^2.$$
\end{prop}
\begin{proof} By Green's formula, using $\Delta(g^2)=2g\Delta g+2\Gamma(g),$
\begin{eqnarray*}&&\frac{1}{2}\sum_{x,y\in V}w_{xy}|\nabla_{xy} (fg)|^2\\&=&-\sum fg\Delta(fg)m=-\sum fg(g\Delta f +f\Delta g+2\Gamma(f,g))m\\
&=& \sum f^2\Gamma(g)m-\sum f\Delta f g^2m-\frac12\sum f^2\Delta g^2m-\sum 2fg\Gamma(f,g)m.
\end{eqnarray*}

By Green's formula and the symmetrization, 
\begin{eqnarray*}
&&-\frac12\sum f^2\Delta g^2m-\sum 2fg\Gamma(f,g)m\\
&=&\frac{1}{4}\sum_{x,y}w_{xy}\nabla_{xy}f^2\nabla_{xy}g^2-\sum_{x,y}w_{xy}f(x)g(x)\nabla_{xy}f\nabla_{xy}g\\
&=&
\frac{1}{4}\sum_{x,y}w_{xy}\nabla_{xy}f^2\nabla_{xy}g^2-\frac12\sum_{x,y}w_{xy}(f(x)g(x)+f(y)g(y))\nabla_{xy}f\nabla_{xy}g\\
&=&-\frac14\sum_{x,y}w_{xy}|\nabla_{xy} f|^2|\nabla_{xy} g|^2.
\end{eqnarray*} This proves the proposition.
\end{proof}

\begin{prop}\label{prop:less1} Let $G=(V,E,m,w)$ be a weighted graph and $\Omega\subset V$ contain two vertices $x_1,x_2$ satisfying $x_1\sim x_2.$  Suppose that the combinatorial distance $d$ is an intrinsic metric on $G,$ then $\mu_1(\Omega)<1.$
\end{prop}
\begin{proof} By the monotonicity, $\mu_1(\Omega)\leq \mu_1(\{x_1,x_2\}).$ Let $f:=1_{\{x_1\}}.$ Since the first eigenfunction is positive on $\{x_1,x_2\},$ $f$ is not the first eigenfunction of the Laplacian on $\{x_1,x_2\}$ with Dirichlet boundary condition. This yields that
$$\mu_1(\{x_1,x_2\})<\frac{\frac12\sum_{x,y}w_{xy}|\nabla_{xy}f|^2}{\sum_{x}f^2(x)m_x}=\frac{\sum_{y}w_{x_1y}}{m_{x_1}}\leq 1.$$ This proves the proposition. 
\end{proof}



Let $\dist$ be an intrinsic metric.
We denote by $$\mathrm{Lip}_\dist h:=\sup_{x,y\in V:x\neq y}\frac{|h(x)-h(y)|}{\dist(x,y)}$$ the Lipschitz constant of $h$ over the graph. Hence for any $x,y\in V,$
$$|h(x)-h(y)|\leq (\mathrm{Lip}_\dist h) \dist(x,y).$$
For a one-variable function $\tilde{f}$ on $[0, \infty),$ by setting $f(x)=\tilde{f}(r(x)),$ one obtains that
\begin{equation}\label{eq:lip1}\mathrm{Lip}_\dist(f)\leq \mathrm{Lip}(\tilde{f}),\end{equation} where $\mathrm{Lip}(\tilde{f})$ is the Lipschitz constant of $\tilde{f}.$ In fact, for any $x,y\in V,$
$$|f(x)-f(y)|=|\tilde{f}(r(x))-\tilde{f}(r(y))|\leq \mathrm{Lip}(\tilde{f})|r(x)-r(y)|\leq \mathrm{Lip}(\tilde{f})\dist(x,y).$$

\begin{prop}\label{prop:ba1} Suppose that the combinatorial distance $d$ is an intrinsic metric. For any functions $f,h: V\to \R$ and any finite $\Omega\subset V,$  $$\sum_{x,y\in \Omega}w_{xy}(f^2(x)+f^2(y))e^{h(x)+h(y)}\leq 2e^{\mathrm{Lip}_d h}\sum_{x\in \Omega}f^2(x)e^{2h(x)}m_x.$$
\end{prop}
\begin{proof} By the symmetry,
\begin{eqnarray*}&&\sum_{x,y\in \Omega}w_{xy}(f^2(x)+f^2(y))e^{h(x)+h(y)}=2\sum_{x,y\in\Omega}w_{xy}f^2(x)e^{h(x)+h(y)}\\
&=&2\sum_{x,y\in \Omega}w_{xy}f^2(x)e^{2h(x)}e^{h(y)-h(x)}
\leq2e^{\mathrm{Lip}_d h}\sum_{x\in \Omega}f^2(x)e^{2h(x)}m_x.\end{eqnarray*}
\end{proof}

\begin{prop}\label{prop:weaker} Condition \eqref{eq:disa1} implies that Condition \eqref{eq:decay1}.
\end{prop}
\begin{proof}
 By \eqref{eq:disa1}, for some $i_0\in \N$ and any large $R,$
$$\frac{1}{R}\sum_{i=i_0}^{[\frac{R}{4s}]-1}\sum_{x\in \Pi_{4is}^{4is+3s}}f^2 e^{-2ar}m\leq \frac{1}{R}\sum_{\Pi_R}f^2 e^{-2ar}m\to 0,\quad R\to\infty.$$
Therefore, there exist $R_i\to \infty, \ i\to\infty,$ such that
$$\sum_{x\in \Pi_{R_i}^{R_i+3s}}f^2 e^{-2ar}m\to 0, \quad i\to \infty.$$ This proves the proposition.
\end{proof}

\section{Proof of the main theorem}\label{sec:proof}
Let $\Pi$ be an end of $G.$ Let $R_0\geq \max_{x\in \partial \Pi}r(x)$ and $a>0.$
For any $R,L\in \R_+$ satisfying \begin{equation}\label{eq:basic1}R\geq R_0+3s\quad \mathrm{and}\quad L\geq R+3s,\end{equation} we define two functions on $[0,\infty),$ 
\begin{equation}\label{12-1}
\tilde{\phi}(r)=\left\{\begin{array}{ll}0,& r\leq R_0+s,\\
\frac1s (r-R_0-s),& R_0+s<r\leq R_0+2s,\\
1,& R_0+2s<r\leq L+s,\\
\frac1s (L+2s-r),& L+s< r\leq L+2s,\\
0,&  r> L+2s,
\end{array}\right.\end{equation}

\begin{equation}\label{12-2}
\tilde{h}(r)=\left\{\begin{array}{ll} ar,&r\leq R-s,\\
a(R-s),& R-s< r\leq R+4s,\\
-a r+2aR+3as,&  r> R+4s.
\end{array}\right.\end{equation}

 Then we define the test function on $V$ as
\begin{equation}\label{def:ph} \phi \, e^h\quad\mathrm{for}\ \phi(x)=\tilde{\phi}(r(x)),\ h(x)=\tilde{h}(r(x)).\end{equation}
By \eqref{eq:lip1},
\begin{equation}\label{eq:lipcon1}\mathrm{Lip}_\dist(\phi)\leq \frac1s,\ \mathrm{Lip}_\dist(h)\leq a.
\end{equation} The same construction of cut-off functions works for $\dist=d.$


\begin{prop}\label{prop:ba2} Let $\dist$ be a metric on $V.$ For any $x\sim y,$
\begin{equation}\label{eq:bb1}|\nabla_{xy}e^h|^2\leq \frac{(e^{a s}-1)^2}{s^2}\min\{e^{2h(x),2h(y)}\}\dist^2(x,y).\end{equation}

If $\rho=d,$ then for any $x\sim y,$
\begin{equation}\label{eq:kes1}|\nabla_{xy}e^h|^2\leq q(a) (e^{2h(x)}+e^{2h(y)}),\end{equation} where $q(a)=\frac{(e^{a}-1)^2}{e^{2a}+1}.$

\end{prop}
\begin{rem} The second assertion was proved by Fujiwara \cite{Fujiwaragrowth}.
\end{rem}

\begin{proof} Without loss of generality, we may assume that $\dist(x,y)>0.$ Otherwise $r(x)=r(y),$ the results are trivial. We further assume that $h(x)\leq h(y).$

 For the first assertion, since $|h(y)-h(x)|\leq a\dist(x,y)$ and $x\sim y,$ 
\begin{eqnarray*} \frac{(e^{h(y)-h(x)}-1)^2}{\dist^2(x,y)}&\leq&\frac{(e^{a\dist(x,y)}-1)^2}{\dist^2(x,y)}\leq \frac{(e^{as}-1)^2}{s^2},
\end{eqnarray*} where we used the monotonicity of the function $t\mapsto \frac{(e^{at}-1)^2}{t^2}$ in the last inequality. This proves the first assertion.

For the second assertion, since $|h(y)-h(x)|\leq a$ and $x\sim y,$ 
\begin{eqnarray*} \frac{(e^{h(y)-h(x)}-1)^2}{e^{2(h(y)-h(x))}+1}&\leq&\frac{(e^{a}-1)^2}{e^{2a}+1},
\end{eqnarray*} where we used the monotonicity of the function $t\mapsto \frac{(e^{t}-1)^2}{e^{2t}+1}$ in the last inequality. This proves the second assertion.
\end{proof}




Now we are ready to prove Theorem~\ref{thm:main1}.
\begin{proof}[Proof of Theorem~\ref{thm:main1}]
We first prove the result for $\rho\neq d.$
Let $R_0\geq \max_{x\in \partial \Pi}r(x),$ and $R,L\in \R_+$. Let  the functions $\phi$ and $h$ be defined as in \eqref{12-1} and ~\eqref{12-2}. For our purposes, we re-define $\phi$ as $\phi 1_{\Pi}$ such that $\phi|_{V\setminus \Pi}=0.$ Let $\mu_1$ be  the bottom of the spectrum on $\Pi$.
Then by Proposition \ref{prop:form1}
\begin{eqnarray}&&\mu_1\sum f^2 \phi^2 e^{2h}m\leq \frac{1}{2}\sum_{x,y\in V}w_{xy}|\nabla_{xy} (\phi e^hf)|^2\nonumber\\
&=&-\sum f \Delta f \phi^2 e^{2h}m+\sum f^2\Gamma(\phi e^h)m-\frac14\sum_{x,y}w_{xy}|\nabla_{xy}f|^2|\nabla_{xy}(\phi e^h)|^2\nonumber\\
&=&-\sum f \Delta f \phi^2 e^{2h}m+\frac12\sum_{x,y}w_{xy} f^2(x)|\nabla_{xy}\phi|^2e^{2h(y)}\label{eq:t21}\\&&+\sum_{x,y}w_{xy}f^2(x) e^{h(y)}\phi(x)\nabla_{xy}\phi\nabla_{xy}e^h\nonumber\\
&&+\frac12\sum_{x,y}w_{xy} f^2(x)|\nabla_{xy}e^h|^2\phi^2(x)-\frac14\sum_{x,y}w_{xy}|\nabla_{xy}f|^2|\nabla_{xy}(\phi e^h)|^2\nonumber\\
&=:&I+II+III+IV+V,\nonumber\end{eqnarray}
where we have used \eqref{eq:leibniz}.

We estimate the right hand side of the inequality term by term. By \eqref{eq:lip1},
for any $x,y\in V,$
$$|\phi(y)-\phi(x)|\leq \frac1s\dist(x,y),\quad |h(y)-h(x)|\leq a\dist(x,y).$$

Note that for any $x\not\in \Pi_{R_0}^{R_0+3s}\cup \Pi_{L}^{L+3s}$ and $y\sim x,$ we have $\nabla_{xy}\phi=0.$ For the term $II,$ 
\begin{eqnarray} II&= &\frac{1}{2}\sum_{x,y}w_{xy} f^2(x)|\nabla_{xy}\phi|^2e^{2h(x)}e^{2h(y)-2h(x)}
\leq \frac{e^{2as}}{2}\sum_{x,y}w_{xy} f^2(x)|\nabla_{xy}\phi|^2e^{2h(x)}\nonumber\\
&=&\frac{e^{2as}}{2}\Bigg(\sum_{x\in \Pi_{R_0}^{R_0+3s}}+\sum_{x\in \Pi_{L}^{L+3s}}\Bigg) f^2(x)e^{2h(x)}\sum_yw_{xy}|\nabla_{xy}\phi|^2\nonumber\\
&\leq&\frac{e^{2as}}{2s^2}\Bigg(\sum_{x\in \Pi_{R_0}^{R_0+3s}}+\sum_{x\in \Pi_{L}^{L+3s}}\Bigg) f^2(x)e^{2h(x)}\sum_yw_{xy}\dist^2(x,y)\nonumber\\
&\leq&\frac{e^{2as}}{2s^2}\Bigg(\sum_{x\in \Pi_{R_0}^{R_0+3s}}+\sum_{x\in \Pi_{L}^{L+3s}}\Bigg) f^2(x)e^{2h(x)}m_x.\label{eq:t22}
\end{eqnarray}

For the term $III,$ by \eqref{eq:lipcon1} and \eqref{eq:bb1},
\begin{eqnarray}III&\leq&e^{as}\sum_{x,y}w_{xy}f^2(x) e^{h(x)}|\nabla_{xy}\phi||\nabla_{xy}e^h|\nonumber\\
&\leq&e^{as}\Bigg(\sum_{x\in \Pi_{R_0}^{R_0+3s}}+\sum_{x\in \Pi_{L}^{L+3s}}\Bigg)f^2(x) e^{h(x)}\sum_{y}w_{xy}|\nabla_{xy}\phi||\nabla_{xy}e^h|\nonumber\\
&\leq&\frac{e^{as}(e^{as}-1)}{s^2}\Bigg(\sum_{x\in \Pi_{R_0}^{R_0+3s}}+\sum_{x\in \Pi_{L}^{L+3s}}\Bigg)f^2(x) e^{2h(x)}\sum_{y}w_{xy}\dist^2(x,y)\nonumber\\
&\leq&\frac{e^{as}(e^{as}-1)}{s^2}\Bigg(\sum_{x\in \Pi_{R_0}^{R_0+3s}}+\sum_{x\in \Pi_{L}^{L+3s}}\Bigg)f^2(x)e^{2h(x)}m_x\label{eq:t23}.
\end{eqnarray}


Note that $h(x)$ is constant on $\Pi_{R-s}^{R+4s}.$ Hence for any $x\sim y,$ if $x\in \Pi_R^{R+3s}$ or $y\in \Pi_R^{R+3s},$ then $\nabla_{xy}h=0,$ which yields $\nabla_{xy}e^h=0.$
 By \eqref{eq:bb1},
\begin{eqnarray*}
IV&=&\frac12\sum_{x,y\in V:x,y\not\in \Pi_{R}^{R+3s}} w_{xy}|\nabla_{xy}e^h|^2f^2(x)\phi^2(x)\nonumber\\
&\leq&\frac{(e^{as}-1)^2}{2s^2}\sum_{x\not\in \Pi_{R}^{R+3s}}f^2(x)\phi^2(x) e^{2h(x)}\sum_{y}w_{xy}\dist^2(x,y)\nonumber\\
&\leq&\frac{(e^{as}-1)^2}{2s^2}\sum_{x\not\in \Pi_{R}^{R+3s}}f^2(x)\phi^2(x) e^{2h(x)}m_x.\nonumber
\end{eqnarray*}

Noting that the fifth term $V\leq 0,$ combining all estimates above and $\Delta f\geq-\mu f$, we obtain that
\begin{eqnarray*}
\mu_1\sum f^2 \phi^2 e^{2h}m&\leq&\mu\sum f^2\phi^2 e^{2h}m+C(a,s)\Bigg(\sum_{x\in \Pi_{R_0}^{R_0+3s}}+\sum_{x\in \Pi_{L}^{L+3s}}\Bigg)f^2 e^{2h}m\\
&&+p(a,s)\sum_{x\not\in \Pi_{R}^{R+3s}}f^2 \phi^2 e^{2h}m,\\
\end{eqnarray*} where $p(a,s)=\frac{(e^{as}-1)^2}{2s^2},\ C(a,s)=\frac{e^{2as}}{2s^2}+\frac{e^{as}(e^{as}-1)}{s^2}.$
 By choosing $a$ such that $p(a,s)=\mu_1-\mu$ and cancelling terms on both sides, we get that
\begin{eqnarray}\label{eq:t11}
&&(\mu_1-\mu)e^{-8as}\sum_{x\in \Pi_{R}^{R+3s}} f^2 e^{2ar}m\nonumber\\&\leq&(\mu_1-\mu)\sum_{x\in \Pi_{R}^{R+3s}} f^2 \phi^2 e^{2h}m\leq C\Bigg(\sum_{x\in \Pi_{R_0}^{R_0+3s}}+\sum_{x\in \Pi_{L}^{L+3s}}\Bigg)f^2 e^{2h}m\nonumber\\
&=&C\sum_{x\in \Pi_{R_0}^{R_0+3s}}f^2 e^{2ar}m+Ce^{4aR+6as}\sum_{x\in \Pi_{L}^{L+3s}}f^2 e^{-2ar}m
\end{eqnarray} 
By the assumption \eqref{eq:decay1}, choosing $L=R_i$ in \eqref{eq:t11} and letting $i\to \infty,$ we have for $R\geq R_0+3s,$
$$\sum_{x\in \Pi_{R}^{R+3s}} f^2 m\leq \frac{Ce^{8as}}{\mu_1-\mu }e^{-2aR}\sum_{x\in \Pi_{R_0}^{R_0+3s}}f^2 e^{2ar}m\leq \frac{7e^{10as}}{s^2(\mu_1-\mu)}e^{-2aR}\sum_{x\in \Pi_{R_0}^{R_0+3s}}f^2 e^{2ar}m.$$ This proves the result for $\rho\neq d.$

Now we consider the case that $\rho=d.$ Note that $s=1.$ Following the same arguments as above, we obtain that \eqref{eq:t21},  \eqref{eq:t22}, and \eqref{eq:t23}. For the term $IV$ and $V$ therein, we give refined estimates as follows.
By the symmetrization,
\begin{eqnarray}
&&IV+V\nonumber\\&=&\frac14\sum_{x,y}w_{xy}(|\nabla_{xy}e^h|^2(f^2(x)\phi^2(x)+f^2(y)\phi^2(y))-|\nabla_{xy}f|^2|\nabla_{xy}(\phi e^h)|^2)\nonumber\\
&=&\frac12\sum_{x,y}w_{xy}|\nabla_{xy}e^h|^2\phi^2(x)f(x)f(y)+\frac14\sum_{x,y}w_{xy}|\nabla_{xy}e^h|^2f^2(y)\nabla_{xy}(\phi^2)\nonumber\\
&&{+}\frac14\sum_{x,y}w_{xy}|\nabla_{xy}f|^2(|\nabla_{xy}e^h|^2\phi^2(x)-|\nabla_{xy}(\phi e^h)|^2)\nonumber\\
&=:&D_1+A_1+B_1\label{eq:first1}.
\end{eqnarray}

For the term $A_1,$ 
\begin{eqnarray*}A_1&\leq& \frac12\sum_{x,y}w_{xy}|\nabla_{xy}e^h|^2f^2(y)|\nabla_{xy}\phi|\leq \frac12(e^a-1)^2\sum_{x,y}w_{xy}e^{2h(y)}f^2(y)|\nabla_{xy}\phi|\\
&\leq&\frac12(e^a-1)^2\Bigg(\sum_{y\in \Pi_{R_0}^{R_0+3}}+\sum_{y\in \Pi_{L}^{L+3}}\Bigg)e^{2h(y)}f^2(y)m_y
\end{eqnarray*}

For the term $B_1,$ by \eqref{eq:leibniz},
\begin{eqnarray*}B_1&=& \frac14\sum_{x,y}w_{xy}|\nabla_{xy}f|^2\nabla_{xy}\phi(2e^{h(x)+h(y)}\phi(x)-(\phi(x)+\phi(y))e^{2h(y)})\\
&\leq&(1+e^a)\sum_{x,y}w_{xy}(f^2(x)+f^2(y))|\nabla_{xy}\phi| e^{h(x)+h(y)}\\
&=&(1+e^a)\Bigg(\sum_{x,y\in \Pi_{R_0}^{R_0+3}}+\sum_{x,y\in \Pi_{L}^{L+3}}\Bigg)w_{xy}(f^2(x)+f^2(y)) e^{h(x)+h(y)}\\
&\leq &2(1+e^a)e^a\Bigg(\sum_{x\in \Pi_{R_0}^{R_0+3}}+\sum_{x\in \Pi_{L}^{L+3}}\Bigg)f^2(x) e^{2h(x)}m_x,
\end{eqnarray*} where the last inequality follows from Proposition~\ref{prop:ba1}.

For the term $D_1,$ by Proposition~\ref{prop:ba2},
\begin{eqnarray*}&&D_1\\&=&\frac12\sum_{x,y\in V, x,y\not\in \Pi_{R}^{R+3}}w_{xy}|\nabla_{xy}e^h|^2\phi^2(x)f(x)f(y)\\
&\leq& \frac{q}{2}\sum_{x,y\not\in \Pi_{R}^{R+3}}w_{xy}\phi^2(x)f(x)f(y)(e^{2h(x)}+e^{2h(y)})\\
&\leq &\frac{q}{2}\sum_{x,y\not\in \Pi_{R}^{R+3}}w_{xy}\phi^2(x)f(x)f(y)e^{2h(x)}+\frac{q}{2}\sum_{x,y\not\in \Pi_{R}^{R+3}}w_{xy}\phi^2(y)f(x)f(y)e^{2h(y)}\\
&&+\frac{q}{2}\sum_{x,y\not\in \Pi_{R}^{R+3}}w_{xy}\nabla_{yx}(\phi^2)f(x)f(y)e^{2h(y)}
\\
&=&q\sum_{x,y\not\in \Pi_{R}^{R+3}}w_{xy}\phi^2(x)f(x)f(y)e^{2h(x)}+\frac{q}{2}\sum_{x,y\not\in \Pi_{R}^{R+3}}w_{xy}\nabla_{yx}(\phi^2)f(x)f(y)e^{2h(y)}\\
&=:&A_2+B_2,
\end{eqnarray*} where $q=q(a).$

For the term $A_2,$
\begin{eqnarray*}A_2&=&q(a)\sum_{x\not\in \Pi_{R}^{R+3}}\phi^2(x)f(x)e^{2h(x)}\sum_{y\not\in \Pi_{R}^{R+3}} w_{xy}f(y)\\
&\leq&q(a)\sum_{x\not\in \Pi_{R}^{R+3}}\phi^2(x)f(x)e^{2h(x)}\sum_{y} w_{xy}f(y)\\
&\leq&q(a)\sum_{x\not\in \Pi_{R}^{R+3}}\phi^2(x)f(x)e^{2h(x)}m_x(\Delta f(x)+f(x)).
\end{eqnarray*}

For the term $B_2,$ by Proposition~\ref{prop:ba1},
\begin{eqnarray*}
B_2&\leq&
q(a)\sum_{x,y}w_{xy}|\nabla_{xy}\phi|f(x)f(y)e^{2h(y)}\\
&\leq& \frac{q(a)}{2}e^a\sum_{x,y}w_{xy}|\nabla_{xy}\phi|(f^2(x)+f^2(y))e^{h(x)+h(y)}\\
&\leq&\frac{q(a)}{2}e^a\Bigg(\sum_{x,y\in \Pi_{R_0}^{R_0+3}}+\sum_{x,y\in \Pi_{L}^{L+3}}\Bigg)w_{xy}(f^2(x)+f^2(y))e^{h(x)+h(y)}\\
&\leq&q(a)e^{2a}\Bigg(\sum_{x\in \Pi_{R_0}^{R_0+3}}+\sum_{x\in \Pi_{L}^{L+3}}\Bigg)f^2(x) e^{2h(x)}m_x.
\end{eqnarray*}

Choose $a$ such that $q(a)=\frac{\mu_1-\mu}{1-\mu}.$ By Proposition~\ref{prop:less1}, $\mu_1<1,$ which yields that $q(a)<1.$
Combining all the estimates above with \eqref{eq:t21},  \eqref{eq:t22} and \eqref{eq:t23}, by $\Delta f\geq -\mu f$, we get
\begin{eqnarray*}
&&\mu_1\sum f^2 \phi^2 e^{2h}m\\&\leq&-\sum f \Delta f \phi^2 e^{2h}m+C(a)\left(\sum_{x\in \Pi_{R_0}^{R_0+3}}+\sum_{x\in \Pi_{L}^{L+3}}\right)f^2 e^{2h}m\\
&&+q(a)\sum_{x\not\in \Pi_{R}^{R+3}}\phi^2(x)f(x)e^{2h(x)}m_x(\Delta f(x)+f(x))\\
&=&-(1-q(a))\sum_{x\not\in \Pi_{R}^{R+3}}f \Delta f \phi^2 e^{2h}m- \sum_{x\in \Pi_{R}^{R+3}}f \Delta f \phi^2 e^{2h}m+q(a)\sum_{x\not\in \Pi_{R}^{R+3}}f^2 \phi^2 e^{2h}m\\
&&+C(a)\Bigg(\sum_{x\in \Pi_{R_0}^{R_0+3}}+\sum_{x\in \Pi_{L}^{L+3}}\Bigg)f^2 e^{2h}m,\\
&\leq&((1-q(a))\mu+q(a))\sum_{x\not\in \Pi_{R}^{R+3}}f^2 \phi^2 e^{2h}m+\mu\sum_{x\in \Pi_{R}^{R+3}}f^2 \phi^2 e^{2h}m\\
&&+C(a)\Bigg(\sum_{x\in \Pi_{R_0}^{R_0+3}}+\sum_{x\in \Pi_{L}^{L+3}}\Bigg)f^2 e^{2h}m
\end{eqnarray*} where, for $q(a)<1,$ $$C(a)=\frac{e^{2a}}{2}+e^a(e^a-1)+\frac12(e^a-1)^2+2e^a(1+e^a)+q(a)e^{2a}\leq 7e^{2a}.$$
Noting that $(1-q(a))\mu+q(a)=\mu_1$ and cancelling terms on both sides, 
we obtain that
\begin{eqnarray}
&&(\mu_1-\mu)e^{2a R-2as}\sum_{x\in \Pi_{R}^{R+3}} f^2m\nonumber\\ &\leq&(\mu_1-\mu)\sum_{x\in \Pi_{R}^{R+3}} f^2 \phi^2 e^{2h}m\nonumber\\&\leq&C(a)\Bigg(\sum_{x\in \Pi_{R_0}^{R_0+3}}+\sum_{x\in \Pi_{L}^{L+3}}\Bigg)f^2 e^{2h}m\nonumber\\
&=&C(a)\sum_{x\in \Pi_{R_0}^{R_0+3}}f^2 e^{2ar}m+C(a)e^{4aR+6a}\sum_{x\in \Pi_{L}^{L+3}}f^2 e^{-2ar}m.\label{eq:mainest11}
\end{eqnarray}

By the assumption \eqref{eq:decay1}, choosing $L=R_i$ in \eqref{eq:mainest11} and letting $i\to \infty,$ we have for $R\geq R_0+3,$
$$\sum_{x\in \Pi_{R}^{R+3}} f^2m \leq \frac{C(a) e^{2a}}{\mu_1-\mu }e^{-2aR}\sum_{x\in \Pi_{R_0}^{R_0+3}}f^2 e^{2ar}m\leq \frac{7e^{10a}}{\mu_1-\mu }e^{-2aR}\sum_{x\in \Pi_{R_0}^{R_0+3}}f^2 e^{2ar}m.$$ This proves the result for $\rho=d.$

Hence we prove the theorem.

\end{proof}

\section{Applications}\label{sec:app}
In this section, we investigate various applications of the main result.

\begin{lemma}\label{lem:keyest} Let $G=(V,E,m,w)$ be an infinite weighted graph, and $\dist$ be an intrinsic metric satisfying Assumption~\ref{ass:1}. Let $\Pi$ be an end such that $\mu_1(\Pi)>0$ and $a=a_0(\mu_1(\Pi)),$ {where $a_0(\mu_1(\Pi))$ is defined in Definition~\ref{def:d1}. } Let $R_0\geq \sup_{x\in \partial \Pi}r(x).$
 For any $R\geq R_0,$ let $f_R$ satisfy
\[\left\{\begin{array}{ll} \Delta f_R(x)=0,& x\in \Pi_R,\\
f_R\big|_{\Pi\setminus \Pi_R}=0.& 
\end{array}\right.\]  If there exist $R_i\to \infty,$ such that $f=\lim_{i\to\infty}f_{R_i}$ on $\Pi.$ Then for $R\geq R_0+3s,$
$$\sum_{\Pi_{R}^{R+3s}}f^2m\leq C e^{-2aR} \sum_{\Pi_{R_0}^{R_0+3s}}f^2e^{2ar}m.$$ 
\end{lemma}
\begin{proof} 
Note that $|f_{R}|$ is a subharmonic function on $\Pi,$ which satisfies the assumptions of Theorem~\ref{thm:main1}. Hence for any $R\geq R_0+3s,$ 
$$\sum_{\Pi_{R}^{R+3s}}f_{R_i}^2m\leq C e^{-2aR} \sum_{\Pi_{R_0}^{R_0+3s}}f_{R_i}^2e^{2ar}m.$$ 
By passing to the limit, $i\to \infty,$ we prove the lemma.
\end{proof}

Now we are ready to prove Theorem~\ref{thm:app1}.
\begin{proof}[Proof of Theorem~\ref{thm:app1}]
For any $\alpha\geq 0,$ we set $$g_{\alpha,R}(x,y)=\int_0^\infty e^{-\alpha t}p_{t,R}(x,y)dt,$$ where $p_{t,R}(x,y)$ denotes the heat kernel on $B_R(x_0)$ with Dirichlet boundary condition, which is zero if $x\in V\setminus B_R(x_0)$ or $y\in V\setminus B_R(x_0).$ Note that $g_{\alpha,R}(x,y)$ is non-decreasing in $R,$ and \begin{equation}\label{eq:ddq}g_\alpha(x_0,x)=\lim_{R\to \infty}g_{\alpha,R}(x_0,x),\quad \forall x\in V.\end{equation}
For $\alpha>0,$ $g_\alpha(x_0,x)<\infty$ for any $x\in V.$ For $\alpha=0,$ since $\mu_1(G)>0,$ the graph is non-parabolic, i.e. the minimal Green's function $g_0(x_0,x)<\infty$ for all $x\in V.$ 

Let $\Omega=\{x_0\}.$ By the local finiteness, there are only finitely many ends $\{\Pi_j\}_{j=1}^J$ with respect to $\Omega.$ Noting that for any $1\leq j\leq J,$ $\mu_1(\Pi_j)\geq \mu_1(G)>0,$ by \eqref{eq:ddq} and the same argument as in Lemma~\ref{lem:keyest}, we have the estimate
$$\sum_{{(\Pi_j)}_{R}^{R+3s}}g_{\alpha}^2(x_0,\cdot)m\leq C e^{-2aR},\quad\forall R\geq 4s.$$ The summation over $j$ yields the result for $R\geq 4s,$ since the estimate on finite connected components of $V\setminus \Omega$ is trivial. The result for $R<4s$ is obvious. This proves the theorem.
\end{proof}

For any $R>0,$ let $f_R$ satisfy 
\begin{equation}\label{eq:frc}\left\{\begin{array}{ll} \Delta f_{R}(x)=0,& x\in \Pi_R\\
f_R\big|_{\partial \Pi}=1,&\\
f_R\big|_{\Pi\setminus \Pi_R}=0.&\\
\end{array}\right.\end{equation}  By the monotonicity of $f_R,$ we define \begin{equation}\label{eq:para1}f=\lim_{R\to\infty} f_R.\end{equation} By Definition~\ref{def:para1} and the maximum principle, if $\Pi$ is parabolic, then $f\equiv 1$ on $\Pi.$ Otherwise, $f$ is a barrier function on $\Pi.$ In the latter case, $f$ is the minimal barrier function on $\Pi,$ and satisfies $\liminf_{x\to \Pi(\infty)}f(x)=0.$ Hence, we have the following proposition.
\begin{prop}\label{prop:para}$\Pi$ is a parabolic end if and only if $f\equiv 1$ on $\Pi,$ where $f$ is defined in \eqref{eq:para1}.
\end{prop} One can show that $G$ is non-parabolic if and only if for some (hence for all) finite $\Omega\subset V$ there exists a non-parabolic end with respect to $\Omega.$

\begin{proof}[Proof of Corollary~\ref{thm:volume}]
It is obvious that $(v)\Longrightarrow (iii)\Longrightarrow (ii).$ 

The statement $(ii)\Longrightarrow(iv)$ follows from Theorem~\ref{thm:main1}  by choosing $f\equiv 1$ and $\mu=0.$

For $(iv)\Longrightarrow(v),$ by taking $R$ as $R+3is$ in the inequality in $(iv)$, for $i=0,1,2,\cdots,$ and summing over $i,$ we obtain that, for $R\geq R_0+3s,$
\begin{eqnarray*}|\Pi|-|\Pi_R|&\leq& C \sum_{i=0}^\infty e^{-2a(R+3is-R_0)}|\Pi_{R_0}^{R_0+3s}|\\
&\leq & Ce^{-2a(R-R_0)}|\Pi_{R_0}^{R_0+3s}|.
\end{eqnarray*} This yields the result.

Now we show that $(i)\Longrightarrow(iv).$ 
Let $f_R$ be the harmonic function on $\Pi_R$ satisfying \eqref{eq:frc} and $\lim_{R\to \infty}f_{R}=f.$ Since $\Pi$ is a parabolic end, $f\equiv 1.$ Then by Lemma~\ref{lem:keyest}, for any $R\geq R_0+3s,$ $$|\Pi_R^{R+3s}|\leq C e^{-2a(R-R_0)}|\Pi_{R_0}^{R_0+3s}|.$$ This proves the statement $(iv).$

At last, we prove that $(iii)\Longrightarrow (i).$ Suppose that $\Pi$ has finite total volume. Let $f_R$ be the harmonic function on $\Pi_R$ satisfying \eqref{eq:frc} and $\lim_{R\to \infty}f_{R}=f.$ We want to show that $f\equiv 1$ on $\Pi,$ which would imply that $\Pi$ is a parabolic end. By the maximum principle, $\max_{x\in \Pi}|f(x)|\leq 1.$ Consider the function $\tilde{f}=f-1.$ Note that $\tilde{f}$ is bounded and $|\Pi|<\infty,$ $\tilde{f}\in \ell^2(\Pi,m).$
Moreover, $\tilde{f}\big|_{\partial \Pi}\equiv 0.$ By a discrete analog of Yau's $\ell^2$ Liouville theorem on graphs, see \cite{HuaKeller14,HuaJost11}, with minor modification for the Dirichlet boundary condition on $\partial\Pi,$ we prove that $\tilde{f}$ is constant on $\overline{\Pi},$ which implies that $f\equiv 1$ on $\overline{\Pi}.$ This yields that $\Pi$ is parabolic, and hence the statement $(i)$ holds.



This proves the corollary.

\end{proof}

\begin{proof}[Proof of Corollary~\ref{thm:volume2}]
Note that $(i),(ii),(iii)$ in this corollary are converse statements of $(i),(ii),(iii)$ in Corollary~\ref{thm:volume}.
 Hence they are equivalent.
 
We show that $(ii)\Longrightarrow (iv).$ By $(ii),$ for any $R\geq R_0,$
$$e^{-2a R} |\Pi_R^{R+3s}|\geq\sum_{\Pi_{R}^{R+3s}}e^{-2a r}m\geq \epsilon_0.$$ This yields the statement $(iv).$

Moreover, $(iv)\Longrightarrow (iii)$ is trivial. Hence all properties $(i)-(iv)$ are equivalent. This proves the corollary.

\end{proof}

\begin{proof}[Proof of Corollary~\ref{coro:app3}]  For any $\epsilon>0$ there exists $R_0(\epsilon)>0$ such that 
$$\mu_1(V\setminus B_{R_0}(x_0))\geq \mu_{e}(G)-\epsilon.$$ By the local finiteness, there are finitely many ends $\{\Pi_i\}_{i=1}^N$ of $G$ with respect to $B_{R_0}(x_0).$ Moreover, for any $1\leq i\leq N,$ $$\mu_1(\Pi_i)\geq \mu_1(V\setminus B_{R_0}(x_0))\geq \mu_{e}(G)-\epsilon.$$ Let $a=a_0( \mu_{e}(G)-\epsilon).$ Since $a_0(\cdot)$ is non-decreasing, for any $1\leq i\leq N,$  $a_0(\mu_1(\Pi_i))\geq a.$

First we consider the case that the total volume of $G$ is infinite.
Hence there exists an end $\Pi_j,$ for some $1\leq j\leq N,$ such that $|\Pi_j|=\infty.$ 
 By Corollary~\ref{thm:volume2},
for sufficiently large $R,$
$$|(\Pi_j)_R|\geq Ce^{2R a_0(\mu_1({\Pi_j}))}\geq Ce^{2aR},$$ where $C$ is independent of $R.$ This yields that
$$|B_R(x_0)|\geq |(\Pi_j)_R|\geq Ce^{2aR}.$$ Therefore,
$$\frac{\log |B_R(x_0)|}{R}\geq 2 a+\frac{\log C}{R}.$$
By taking the liminf for $R\to \infty,$ we get 
$$\wt{\tau}(G)=\liminf_{R\to\infty}\frac{\log |B_R(x_0)|}{R}\geq 2 a.$$  By passing to the limit $\epsilon \to 0,$ the result follows.

Next we consider the case that the total volume of $G$ is finite. Then $|\Pi_i|<\infty,$ for all $1\leq i\leq N.$
 By Corollary~\ref{thm:volume}, there exist $C_i,$ $1\leq i\leq N,$ such that 
for $R\geq R_0+3s,$
$$|\Pi_i\setminus {(\Pi_i)}_R|\leq C_ie^{-2R a_0(\mu_1({\Pi_i}))}\leq C_ie^{-2aR},\quad 1\leq i\leq N.$$ 
For sufficiently large $R$ such that $\partial B_{R}(x_0)$ has no intersection with finite connected components of $V\setminus B_{R_0}(x_0).$ Summing over $i=1,\cdots,N,$ we obtain that
$$|V\setminus B_R(x_0)| \leq Ce^{-2aR}.$$ Hence 
$$\frac{-\log (|V\setminus B_R(x_0)|)}{R}\geq 2 a-\frac{\log C}{R}.$$ The result follows from taking the liminf for $R\to\infty$ and then letting $\epsilon\to 0.$
\end{proof}

\begin{coro}\label{coro:app4}Let $G=(V,E,m,w)$ be an infinite weighted graph, and $\dist$ be an intrinsic metric satisfying Assumption~\ref{ass:1}.  Let $\Pi$ be an end satisfying $\mu_1(\Pi)>0.$ Let $u:\overline{\Pi}\to \R$ be a harmonic function on $\Pi$ satisfying $u\in \ell^q(\Pi,m)$ for $q>1.$ 
If either 
\begin{enumerate}
\item $q\geq 2,$ or
\item $q\in(1,2)$ and  $|\Pi_R|\leq C e^{\frac{2q}{2-q} aR}$ for sufficiently large $R,$ where $a=a_0(\mu_1(\Pi)),$
\end{enumerate}
then $u$ is bounded and for sufficiently large $R,$ $$\sum_{\Pi_{R}^{R+3s}}u^2m\leq Ce^{-2R a_0(\mu_1(\Pi))}.$$

\end{coro}
\begin{proof} Let $R_0=\max_{x\in \partial \Pi}r(x).$ For any $R\geq R_0+s,$ let $f_{R}$ be the harmonic function on $\Pi_R$ such that $f\big|_{\partial \Pi}=u, f\big|_{\Pi\setminus\Pi_R}=0.$ By the maximum principle,
$|f_R|$ is bounded by $\max_{\partial \Pi}|u|.$ So that there exists a sequence $R_i\to \infty,$ such that
$$f_{R_i}(x)\to f(x),\quad \forall x\in \Pi\cup \partial \Pi.$$ Hence, $f$ is a bounded harmonic function on $\Pi.$ By Lemma~\ref{lem:keyest}, $R\geq R_0+3s,$
\begin{equation}\label{eq:part1}\sum_{\Pi_{R}^{R+3s}}f^2m\leq C e^{-2aR }.\end{equation} 
This yields that $f\in \ell^2(\Pi,m).$ 

For $q\geq 2,$ since $f$ is bounded, $f\in \ell^q(\Pi,m).$ Hence 
$u-f\in \ell^q(\Pi,m),$ which vanishes on $\partial \Pi.$ By a discrete analog of Yau's $\ell^q$ Liouville theorem on graphs, see \cite{HuaKeller14,HuaJost11}, with minor modification for the Dirichlet boundary condition on $\partial\Pi,$ $u-f$ is constant and hence $u-f\equiv 0.$ The result follows from \eqref{eq:part1} for $q\geq 2.$

For $q\in(1,2),$ by \eqref{eq:part1}, the Cauchy-Schwarz inequality and the volume growth bound,
\begin{eqnarray*}
\sum_{\Pi_R^{R+3s}}|f|^qm&\leq&|\Pi_{R}^{R+3s}|^{1-\frac{q}{2}}\left(\sum_{\Pi_R^{R+3s}}|f|^2m\right)^{\frac{q}{2}}\\
&\leq & Ce^{q aR} e^{-q aR}=C.
\end{eqnarray*} Since $u\in \ell^q(\Pi,m),$ $$\sum_{\Pi_R} |u-f|^qm=O(R),\quad R\to\infty.$$
By a discrete analog of quantitative Yau's $\ell^q$ Liouville theorem on graphs, see \cite{HuaKeller14,HuaJost11}, with necessary modification on the Dirichlet boundary condition, one obtains that
$u-f$ is constant and hence $u-f\equiv 0.$ The result follows from \eqref{eq:part1}.

This proves the corollary.

\end{proof}

\bigskip
\bigskip

\textbf{Acknowledgements.} We thank Dr. Lu Li for helpful discussions on the problem.  B. H. is supported by NSFC, no.11831004 and no. 11926313. Z. L. is supported by an NSF grant of  DMS-19-08513 of USA.

On behalf of all authors, the corresponding author states that there is no conflict of interest.

 Data sharing not applicable to this article as no datasets were generated or analysed during the current study.

\bibliographystyle{alpha}
\bibliography{pspec}

\end{document}